\documentclass[12pt]{amsart}

\usepackage{amsmath,amsthm,amssymb,mathdots,enumerate,fullpage}
\newtheorem{theorem}{Theorem}[section]
\newtheorem{lemma}[theorem]{Lemma}
\newtheorem{prop}[theorem]{Proposition}

\newtheorem{cor}[theorem]{Corollary}

\theoremstyle{remark}

\newcommand{\R}{\mathbb R}
\newcommand{\C}{\mathbb C}

\newcommand{\Z}{\mathbb Z}
\newcommand{\Q}{\mathbb Q}

\newcommand{\A}{\mathbb A}

\newcommand{\g}{\mathfrak g}

\newcommand{\cO}{\mathcal O}

\newcommand{\cL}{\mathcal L}

\newcommand{\cS}{\mathcal S}

\newcommand{\cE}{\mathcal E}
\newcommand{\cH}{\mathcal H}

\newcommand{\tPhi}{\widetilde{\Phi}}
\newcommand{\tPsi}{\widetilde{\Psi}}
\newcommand{\tpsi}{\widetilde{\psi}}
\newcommand{\tcE}{\widetilde{\cE}}
\newcommand{\tcH}{\widetilde{\cH}}

\newcommand{\tr}{\text{tr}}

\newcommand{\be}{\begin{equation}}
\newcommand{\ee}{\end{equation}}
\newcommand{\bes}{\begin{equation*}}
\newcommand{\ees}{\end{equation*}}

\newcommand{\ba}{\begin{eqnarray}}
\newcommand{\ea}{\end{eqnarray}}
\newcommand{\bas}{\begin{eqnarray*}}
\newcommand{\eas}{\end{eqnarray*}}

\title{Endoscopy and cohomology of a quasi-split $U(4)$}

\author{Simon Marshall}
\address{Department of Mathematics\\
University of Wisconsin -- Madison\\
480 Lincoln Drive\\
Madison\\
WI 53706, USA}
\email{marshall@math.wisc.edu}
\thanks{Supported by NSF grant DMS-1201321.}

\begin{document}

\begin{abstract}
We prove asymptotic upper bounds for the $L^2$ Betti numbers of the locally symmetric spaces associated to a quasi-split $U(4)$.  These manifolds are 8-dimensional, and we prove bounds in degrees 2 and 3, with the behaviour in the other degrees being well understood.  In degree 3, we conjecture that these bounds are sharp.  Our main tool is the endoscopic classification of automorphic representations of $U(N)$ by Mok.
\end{abstract}

\maketitle

\section{Introduction}

Let $E$ be an imaginary quadratic field.  Let $N \ge 1$, let $U(N)$ be the quasi-split unitary group of degree $N$ with respect to $E/\Q$, and let $G$ be an inner form of $U(N)$.  Let $\Gamma \subset G(\Q)$ be an arithmetic congruence lattice, and for $n \ge 1$ let $\Gamma(n)$ be the corresponding principal congruence subgroup of $\Gamma$.  Let $K_\infty$ be a maximal compact subgroup of $G(\R)$.  Let $Y(n) = \Gamma(n) \backslash G(\R) / K_\infty$, which is a complex orbifold (or manifold if $n$ is large enough).  We let $H^i_{(2)}(Y(n))$ be the $L^2$ cohomology groups of $Y(n)$.  By \cite{BC}, $H^i_{(2)}(Y(n))$ is equal to the space of square-integrable harmonic $i$-forms on $Y(n)$, and we shall identify it with this space from now on.  Note that $H^i_{(2)}(Y(n)) = H^i(Y(n))$ when $Y(n)$ is compact.  We set $h^i_{(2)}(Y(n)) = \dim H^i_{(2)}(Y(n))$.  This article is interested in how $h^i_{(2)}(Y(n))$ grow with $n$, specifically in the case when $G = U(4)$.

We let $V(n) = |\Gamma : \Gamma(n)|$, which is asymptotically equal to the volume of $Y(n)$.  The standard bound that we wish to improve over is $h^i_{(2)}(Y(n)) \ll V(n)$.  This follows from the equality of $h^i_{(2)}$ with an ordinary Betti number if $\Gamma$ is cocompact, and otherwise from the noncompact version of Matsushima's formula in \cite[Prop 5.6]{BG} which expresses $h^i_{(2)}(Y(n))$ in terms of automorphic representations, together with Savin's bound \cite{Sa} for the multiplicity of a representation in the cuspidal spectrum and Langlands' theory of Eisenstein series.

The basic principle that we shall use to bound $h^i_{(2)}(Y(n))$ is the fact that, if $i$ is not half the dimension of $Y(n)$, the archimedean automorphic forms that contribute to $h^i_{(2)}(Y(n))$ must be nontempered.  In the case where $\Gamma$ is cocompact, one may combine this principle with the trace formula and asymptotics of matrix coefficients to prove a bound of the form $h^i_{(2)}(Y(n)) \ll V(n)^{1-\delta}$ for some $\delta > 0$.  In \cite{SX}, Sarnak and Xue suggest the optimal bound that one should be able to prove in this way using only the archimedean trace formula.  In the case when $N = 3$ and $\Gamma$ is cocompact (which implies that $Y(n)$ have real dimension 4), they predict that $h^1_{(2)}(Y(n)) \ll_\epsilon V(n)^{1/2+\epsilon}$, while they prove that $h^1_{(2)}(Y(n)) \ll_\epsilon V(n)^{7/12+\epsilon}$.

There is a deeper way in which one may exploit nontemperedness to prove bounds for cohomology.  In \cite{Mo} Mok, following Arthur \cite{A}, classifies the automorphic spectrum of $U(N)$ in terms of conjugate self-dual cusp forms on $GL_M/E$ for $M \le N$.  One of the implicit features of this classification is that if a representation $\pi$ on $U(N)$ is sufficiently nontempered at one place, then it must be built up from cusp forms on groups $GL_M/E$ with $M$ strictly less than $N$ -- in other words, $\pi$ comes from a smaller group.  We have been interested in deriving quantitative results from this qualitative feature of the classification.  In \cite{Ma}, we used this (more precisely, the complete solution of endoscopy for $U(3)$ by Rogawski in \cite{Ro}) to prove that $h^1_{(2)}(Y(n)) \ll_\epsilon V(n)^{3/8+\epsilon}$ when $N = 3$ and $G$ is arbitrary, strengthening the bound of Sarnak and Xue.  Moreover, we proved that this bound is sharp.  In this article, we partially extend this result to the case $G = U(4)$.  Note that in this case, the real dimension of $Y(n)$ is 8.

\begin{theorem}
\label{simplemainthm}

If $G = U(4)$ and $i = 2$ or $3$, and $n$ is only divisible by primes that split in $E$, we have $h^i_{(2)}(Y(n)) \ll_\epsilon V(n)^{8/15 + \epsilon}$.

\end{theorem}

See Theorem \ref{mainthm} for a precise statement.  We expect Theorem \ref{simplemainthm} to be sharp in the case $i = 3$, but when $i = 2$ we expect the true order of growth to be $V(n)^{2/5 + \epsilon}$ for reasons discussed below.  Note that we have  $h^1_{(2)}(Y(n)) = 0$ for all $n$, by combining the noncompact Matsushima formula of \cite[Prop 5.6]{BG} with the vanishing theorems of e.g. $\mathsection$10.1 of \cite{BW}.  The results of \cite{Sa} also imply that $h^4_{(2)}(Y(n)) \gg V(n)$.

\subsection{Outline of proof}

To describe the method of proof of Theorem \ref{simplemainthm} in more detail, we begin by outlining the classification of Arthur and Mok.  We define an Arthur parameter for $U(N)$ to be a formal linear combination $\psi = \nu(n_1) \boxtimes \mu_1 \boxplus \ldots \boxplus \nu(n_l) \boxtimes \mu_l$, where $\nu(k)$ denotes the unique irreducible (complex-algebraic) representation of $SL(2,\C)$ of dimension $k$, and $\mu_i$ is a conjugate self-dual cusp form on $GL_{m_i}/E$, subject to certain conditions including that $N = \sum n_i m_i$.  To each $\psi$, there is associated a packet $\Pi_\psi$ of representations of $U(N)(\A)$, certain of which occur in the automorphic spectrum.  Moreover, the entire automorphic spectrum is obtained in this way.  If we combine this classification with the noncompact case of Matsushima's formula, we have

\be
\label{Matsueg}
h^i_{(2)}(Y(n)) \le \sum_\psi \sum_{\pi \in \Pi_\psi} h^i( \g, K; \pi_\infty) \dim \pi_f^{K(n)}.
\ee
Here, $\g$ is the Lie algebra of $U(N)(\R)$, $K$ is a maximal compact subgroup of $U(N)(\R)$, we let $H^i(\g,K;\pi_\infty)$ denote $(\g,K)$ cohomology, and $h^i(\g,K;\pi_\infty) = \dim H^i(\g,K;\pi_\infty)$.  As mentioned above, if $i$ is not the middle degree, then those $\psi$ contributing to the sum must be \textit{non-generic}, i.e. one of the representations of $SL(2,\C)$ must be nontrivial.

We deduce Theorem \ref{simplemainthm} from (\ref{Matsueg}) in two steps.

\noindent
{\bf Step 1:} Bound $\sum_{\pi_f \in \Pi_{\psi,f}} \dim \pi_f^{K(n)}$ for each each $\psi$, where $\Pi_{\psi,f}$ denotes the finite part of the packet $\Pi_\psi$.

\noindent
{\bf Step 2:} Sum the resulting bounds over those $\psi$ that contribute to cohomology in the required degree.

We begin step 1 by writing $\Pi_{\psi,f} = \otimes_p \Pi_{\psi,p}$, so that we must bound $\sum_{\pi_p \in \Pi_{\psi,p}} \dim \pi_p^{K(n)}$ for each $p$.  When $p$ is split in $E$, $\Pi_{\psi,p}$ is an explicitly described singleton, and it is easy to do this directly.  When $p$ is nonsplit, we use the trace identities appearing in the definition of $\Pi_{\psi,p}$ \cite[Theorem 3.2.1]{Mo}.  By writing $\dim \pi_p^{K(n)}$ as a trace, these allow us to relate $\sum_{\pi_p \in \Pi_{\psi,p}} \dim \pi_p^{K(n)}$ to objects like $\dim \mu_i^{K'(n)}$, where $\mu_i$ is one of the cusp forms appearing in $\psi$ and $K'(n)$ is a suitable congruence subgroup of $GL(m_i)$.

As an example, one type of packet that contributes to (\ref{Matsueg}) when $N = 4$ is those of the form $\psi = \nu(2) \boxtimes \mu$, where $\mu$ is a cusp form on $GL_2/E$.  After carrying out step 1 in this case, we obtain

\be
\label{U2eg}
\sum_{\pi_f \in \Pi_{\psi,f}} \dim \pi_f^{K(n)} \ll n^{5 + \epsilon} \sum_{\pi'_f \in \Pi(\mu)_f} \dim \pi_f'^{K'(n)}
\ee
where $\Pi(\mu)$ is the packet on $U(2)$ corresponding to $\mu$, and $K'(n)$ is the standard principal congruence subgroup of level $n$ on $U(2)$.  Step 2 is bounding the right hand side of (\ref{U2eg}).  We do this by observing that if $\Pi_\psi$ contains a cohomological representation, and $\pi' \in \Pi(\mu)$ as in (\ref{U2eg}), then there are only finitely many possibilities for the infinitesimal character of $\pi'_\infty$, and hence of $\pi'_\infty$ itself.  We may therefore bound the right hand side of (\ref{U2eg}) in terms of the multiplicities of archimedean representations on $U(2)$, and these may be bounded by the results of Savin.

The reason we do not expect Theorem \ref{simplemainthm} to be sharp when $i = 2$ is that the main contribution to $h^2_{(2)}$ comes from parameters of the form $\nu(2) \boxtimes \mu$ with $\mu$ on $GL_2$.  (Note that this relies on the Adams-Johnson conjectures on the structure of cohomological Arthur packets, which have now been proved by Arancibia, Moeglin, and Renard \cite{AMR}.)  We do not have sharp bounds for the contribution from these parameters, because we do not have sharp bounds for the dimensions of spaces of $K$-fixed vectors in Speh representations induced from $GL_2 \times GL_2$ on $GL_4$.  To be more precise, if $\pi$ is such a Speh representation of $GL(4,\Q_p)$, we require a bound for $\dim \pi^{K(p^k)}$, where $K(p^k)$ is the usual principal congruence subgroup, that is uniform in both $k$ and $\pi$.  In particular, this is more difficult than knowing the Kirillov dimension of these representations.

We have restricted to levels that are split in $E$ because of an issue with the twisted fundamental lemma, which is used in step 1 in the case of inert primes.  Allowing level in this argument would require an extension of the twisted FL, which states that the twisted transfer takes the characteristic functions of principal congruence subgroups to functions of the same type.  This would follow from the twisted FL for Lie algebras, which is not known at this time.  However, it should be possible to prove it by following Waldspurger's proof for groups in \cite{Wa}.

The tools used in the proof should extend to a general $U(N)$ with a little extra work.  However, because the recipe for the degrees of cohomology on $U(N)$ to which an Arthur parameter can contribute is complicated, the result this would give for cohomology growth would not be as strong.  

{\bf Acknowledgements} We would like to thank Tasho Kaletha and Sug Woo Shin for helpful comments.

\section{The endoscopic classification for $U(N)$}
\label{class}

In this section we describe the endoscopic classification for the quasi-split group $U(N)$ by Mok.  Because of the large amount of notation that must be introduced to do this in full, we shall often omit details that are not directly relevant to the proof of Theorem \ref{simplemainthm}.

\subsection{Number fields}
\label{class1}

Throughout this section, $F$ will denote a local or global field of characteristic 0, and $E$ will denote a quadratic \'etale $F$-algebra.  We will assume that $E$ is a quadratic extension of $F$ unless specified otherwise.  The conjugation of $E$ over $F$ will be denoted by $c$.  We set $\Gamma_F = \text{Gal}(\overline{F}/F)$.  The Weil groups of $F$ and $E$ will be denoted by $W_F$ and $W_E$ respectively.  If $F$ is local, we let $L_F$ denote its local Langlands group, which is given by $W_F$ if $F$ is archimedean and $W_F \times SU(2)$ otherwise.  If $F$ is global, the adeles of $F$ and $E$ will be denoted by $\A$ and $\A_E$.  If $F$ is local (resp. global), $\chi$ will denote a character of $E^\times$ (resp. $\A_E^\times / E^\times$) whose restriction to $F^\times$ (resp. $\A^\times$) is the quadratic character associated to $E/F$ by class field theory.  We will often think of $\chi$ as a character of $W_E$.

\subsection{Algebraic groups}
\label{class2}

For any $N \ge 1$, we let $U(N)$ denote the quasi-split unitary group over $F$ with respect to $E / F$, whose group of $F$-points is

\bes
U(N)(F) = \{ g \in GL(N,E) | \; {}^t c(g) J g = J \}
\ees
where

\bes
J = \left( \begin{array}{ccc} && 1 \\ & \iddots & \\ 1 && \end{array} \right).
\ees
In the case when $E = F \times F$, we have

\bes
U(N)(F) = \{ (g_1,g_2) \in GL(N,F) \times GL(N,F) | \; g_2 = J {}^t g_1^{-1} J^{-1} \}.
\ees
Projection onto the first and second factors defines isomorphisms $\iota_1, \iota_2 : U(N)(F) \simeq GL(N,F)$, and we have $\iota_2 \circ \iota_1^{-1} : g \mapsto J {}^t g^{-1} J^{-1}$.

We define $G(N) = \text{Res}_{E/F} GL(N)$.  We let $\theta$ denote the automorphism of $G(N)$ whose action on $F$-points is given by

\bes
\theta(g) = \Phi_N {}^tc(g)^{-1} \Phi_N^{-1} \quad \text{for} \quad g \in G(N)(F) \simeq GL(N,E),
\ees
where

\bes
\Phi_N = \left( \begin{array}{cccc} &&& 1 \\ && -1 & \\ & \iddots && \\ (-1)^{N-1} &&& \end{array} \right).
\ees
We define $\widetilde{G}^+(N) = G(N) \rtimes \langle \theta \rangle$, and let $\widetilde{G}(N)$ denote the $G(N)$-bitorsor $G(N) \rtimes \theta$.  We will denote these groups by $U_{E/F}(N)$, $G_{E/F}(N)$, etc. when we want to explicate the dependence on the extension $E/F$. 

Our discussion in this section will implicitly require choosing Haar measures on the $F$-points of these groups when $F$ is local, in particular when discussing transfers of functions and character relations.  We may do this in an arbitrary way, subject only to the condition that the Haar measures assign mass 1 to a hyperspecial maximal compact subgroup when one exists.  This condition allows us to state the fundamental lemma without the introduction of any constant factors.

\subsection{$L$-groups and embeddings}
\label{class3}

If $G$ is a connected reductive algebraic group over $F$, the $L$-group ${}^LG$ is an extension $\widehat{G} \rtimes W_F$, where $\widehat{G}$ is the complex dual group of $G$. If $G_1$ and $G_2$ are two such groups, an $L$-morphism ${}^L G_1 \to {}^L G_2$ is a map that reduces to the identity map on $W_F$.  An $L$-embedding is an injective $L$-morphism.  In this paper we shall only need to consider ${}^L G$ when $G$ is a product of the groups $U(N)$, $GL(N)$, and $G(N)$.  Because the $L$-group of $G_1 \times G_2$ is the fiber product of ${}^L G_1$ and ${}^L G_2$ over $W_F$, it suffices to specify ${}^L G$ when $G$ is one of these groups. We have ${}^L GL(N) = GL(N,\C) \times W_F$.  We have ${}^LU(N) = GL(N,\C) \rtimes W_F$, where $W_F$ acts through its quotient $\text{Gal}(E/F)$ via the automorphism

\bes
g \mapsto \Phi_N {}^t g^{-1} \Phi_N^{-1}.
\ees
We have ${}^L G(N) = (GL(N,\C) \times GL(N,\C)) \rtimes W_F$, where $W_F$ acts through $\text{Gal}(E/F)$ by switching the two factors.  We let $\widehat{\theta}$ denote the automorphism of $\widehat{G(N)}$ given by $\widehat{\theta}(x,y) = ( \Phi_N {}^t y^{-1} \Phi_N^{-1}, \Phi_N {}^t x^{-1} \Phi_N^{-1} )$.

We define the $L$-embedding $\xi_\kappa : {}^L U(N) \to {}^L G(N)$ for $\kappa = \pm 1$ as follows.  (Note we will often abbreviate $\pm 1$ to simply $\pm$.)  We define $\chi_+ = 1$ and $\chi_- = \chi$, and we choose $w_c \in W_F \setminus W_E$.  We define $\xi_\kappa$ by the following formulae.

\begin{align*}
g \rtimes 1 & \mapsto (g, {}^t g^{-1}) \rtimes 1 \text{ for } g \in GL(N,\C) \\
I \rtimes \sigma & \mapsto (\chi_\kappa(\sigma) I, \chi_\kappa^{-1}(\sigma) I ) \rtimes \sigma \text{ for } \sigma \in W_E \\
I \rtimes w_c & \mapsto (\kappa \Phi_N, \Phi_N^{-1}) \rtimes w_c.
\end{align*}
Note that the conjugacy class of $\xi_\pm$ is independent of the choice of $w_c$.

\subsection{Endoscopic data}
\label{class4}

In the cases we consider in this paper, it suffices to work with a simplified notion of endoscopic datum that we now describe.  See \cite{KS} for the general definition.  Let $G^0$ be a connected reductive group over $F$, and let $\theta$ be a semisimple automorphism of $G^0$.  Let $G$ be the $G^0$-bitorsor $G^0 \rtimes \theta$.  We let $\widehat{\theta}$ be the automorphism of $\widehat{G}^0$ that is dual to $\theta$ and preserves a fixed $\Gamma_F$-splitting of $\widehat{G}_0$.  We shall only need to consider the cases where $\theta$ is trivial or $G$ is the torsor $\widetilde{G}(N)$ defined in Section \ref{class2}, in which cases the dual automorphism $\widehat{\theta}$ is the one given in Section \ref{class3}.

We let $\widehat{G} = \widehat{G}^0 \rtimes \widehat{\theta}$.  An endoscopic datum for $G$ is a triple $(G', s, \xi')$ satisfying the following conditions.

\begin{itemize}

\item $s \in \widehat{G}$ is semi-simple.

\item $G'$ is a quasi-split connected reductive group over $F$.

\item $\xi' : {}^L G' \to {}^L G^0$ is an $L$-embedding.

\item The restriction of $\xi'$ to $\widehat{G}'$ is an isomorphism $\widehat{G}' \simeq \text{Cent}(s, \widehat{G}^0)^0$.

\item We have $\text{Ad}(s) \circ \xi' = a \cdot \xi'$, where $a : W_F \to Z(\widehat{G}^0)$ is a 1-cocycle that is cohomologically trivial if $F$ is local, and is everywhere locally trivial if $F$ is global.

\end{itemize}

We refer to \cite[Section 2.1]{KS} for the definition of equivalence of endoscopic data.  We will often omit the data $s$ and $\xi'$ if they are not immediately relevant.  We say that an endoscopic datum is elliptic if we have

\bes
(Z(\widehat{G}')^{\Gamma_F})^0 \subset Z(\widehat{G}^0)^{\widehat{\theta},\Gamma_F}.
\ees
We denote the set of equivalence classes of endoscopic data for $G$ by $\cE(G)$, and the subset of elliptic data by $\cE_\text{ell}(G)$.  We set $\cE(\widetilde{G}(N)) = \tcE(N)$.  From now on, we shall only use the notation $\cE(G)$ when $G$ is a group, i.e. when $\theta$ is trivial.

There is a subset $\tcE_\text{sim}(N) \subset \tcE_\text{ell}(N)$, called the set of simple endoscopic data, that consists of the elements $(U(N), \xi_+)$ and $(U(N), \xi_-)$ where $\xi_\pm$ are the embeddings of Section \ref{class3}.

\subsection{Transfer of functions}
\label{class5}

From now until the end of Section \ref{class7}, we assume that $F$ is local.  If $G$ is an $F$-group, we denote $C^\infty_0(G)$ by $\cH(G)$.  We denote $C^\infty_0( \widetilde{G}(N))$ by $\tcH(N)$.  If $G$ is a connected reductive group over $F$ and $(G', \xi') \in \cE(G)$, there is a correspondence between $\cH(G)$ and $\cH(G')$ known as the endoscopic transfer.  More precisely, there is a nonempty subset of $\cH(G')$ associated to any $f \in \cH(G)$, and we let $f^{(G',\xi')}$ (which we will often abbreviate to $f^{G'}$) denote a choice of function from it.  We say that $f^{(G',\xi')}$ is an endoscopic transfer of $f$ to $G'$.  The transfer is defined using orbital integrals on $G$ and $G'$ in a way that we do not need to make explicit in this paper.  Its construction is primarily due to Shelstad in the real case, and Waldspurger \cite{Wa3} in the $p$-adic case (assuming the fundamental lemma).  See \cite[Section 2.1]{A} for more details.

We shall require the fundamental lemma, due to Laumon and Ng\^o \cite{LN,N}, Hales \cite{Ha}, Waldspurger \cite{Wa2}, and others.  This states that if the local field $F$ is $p$-adic, all data are unramified, and $K$ and $K'$ are hyperspecial maximal compact subgroups of $G$ and $G'$, then the characteristic functions $1_K$ and $1_{K'}$ correspond under endoscopic transfer.

There is a similar transfer in the twisted case.  If $(G,\xi) \in \tcE(N)$, this associates a function $f^{(G,\xi)} \in \cH(G)$ to a function $f \in \tcH(N)$.  There is a twisted fundamental lemma, derived by Waldspurger in \cite{Wa} from the untwisted case and his nonstandard variant, which states that the characteristic functions of hyperspecial maximal compact subgroups are associated by transfer if $F$ is $p$-adic and all data are unramified.

\subsection{Local parameters}
\label{class6}

Let $G$ be a connected reductive algebraic group over $F$.  A Langlands parameter for $G$ is an admissible homomorphism

\bes
\phi : L_F \rightarrow {}^LG.
\ees
We let $\Phi(G)$ denote the set of Langlands parameters up to conjugacy by $\widehat{G}$.  An Arthur parameter for $G$ is an admissible homomorphism

\bes
\psi : L_F \times SL(2,\C) \rightarrow {}^LG
\ees
such that the image of $L_F$ in $\widehat{G}$ is bounded.  We let $\Psi(G)$ denote the set of Arthur parameters modulo conjugacy by $\widehat{G}$, and let $\Psi^+(G)$ denote the set of parameters obtained by dropping this boundedness condition.

If $\psi \in \Psi^+(G)$ we define the following groups, which control the character identities for the local Arthur packet associated to $\psi$.

\begin{align*}
S_\psi & = \text{Cent}( \text{Im} \psi, \widehat{G} ), \\
\overline{S}_\psi & = S_\psi / Z(\widehat{G})^{\Gamma_F}, \\
\cS_\psi & = \pi_0( \overline{S}_\psi).
\end{align*}
In all cases we consider, we will have $\cS_\psi \simeq (\Z / 2\Z)^r$ for some $r$.  We also define

\bes
s_\psi = \psi\left( 1, \left( \begin{array}{cc} -1 & 0 \\ 0 & -1 \end{array} \right) \right),
\ees
which is a central semi-simple element of $S_\psi$.

\subsubsection{Endoscopic data associated to Arthur parameters}
\label{class61}

There is a correspondence between pairs $(G',\psi')$ with $G' \in \cE(G)$ and $\psi' \in \Psi(G')$, and pairs $(\psi,s)$ with $\psi \in \Psi(G)$ and $s$ a semi-simple element of $\overline{S}_\psi$.  (Note that we place a stronger equivalence relation on $G'$ here than the usual equivalence of endoscopic data; see \cite[Section 3.2]{Mo} for details.)  In one direction, this correspondence associates to a pair $(G',\psi')$ (where $G'$ is an abbreviation of $(G',s',\xi')$) the pair $(\psi,s)$, where $\psi = \xi' \circ \psi'$ and $s$ is the image of $s'$ in $\overline{S}_\psi = S_\psi / Z(\widehat{G})^{\Gamma_F}$.

Conversely, suppose we have a pair $(\psi,s)$.  Let $s'$ be any lift of $s$ to $S_\psi$.  We set $\widehat{G}' = \text{Cent}(s',\widehat{G})^0$.  Because $\psi(W_F)$ commutes with $s'$ it normalises $\widehat{G}'$, and this action allows us to define an $L$-group ${}^L G'$.  We may combine $\psi|_{W_F}$ and the inclusion $\widehat{G}' \subset \widehat{G}$ to obtain an $L$-embedding $\xi' : {}^L G' \to {}^L G$, which gives an endoscopic datum $(G',s',\xi')$.  Because $\psi$ factors through $\xi'({}^L G')$, this gives an $L$-parameter $\psi' \in \Psi(G')$.

\subsubsection{Base change maps}
\label{class62}

We now discuss the map from parameters of $U(N)$ to parameters of $G(N)$ given by $\xi_\pm$.  We first note that there is an isomorphism $\Phi(G(N)) \simeq \Phi(GL(N,E))$, which is given explicitly in \cite[Section 2.2]{Mo}, and corresponds to the fact that both sets parametrize representations of $G(N)(F) \simeq GL(N,E)$.  If $\phi \in \Phi(U(N))$, the parameter in $\Phi(GL(N,E))$ corresponding to $\xi_\pm \circ \phi$ under this isomorphism is just $\phi|_{L_E} \otimes \chi_\pm$.  In particular, in the case of $\xi_+$ the parameter is just obtained by restriction to $L_E$ (this is usually known as the standard base change map).

\subsubsection{Parities of local parameters}
\label{class63}

One may characterise the image of $\Phi(U(N))$ in $\Phi(G(N))$ under $\xi_\pm$.  We say that an admissible homomorphism $\rho : L_E \to GL(N,\C)$ is conjugate self-dual if $\rho^c \simeq \rho^\vee$, where $\rho^c(\sigma) = \rho( w_c^{-1} \sigma w_c)$ for $\sigma \in L_E$ and $w_c \in W_F \setminus W_E$.  There is a notion of parity for a conjugate self-dual representation \cite[Section 2.2]{Mo}, which is analogous to a self-dual representation being either orthogonal (even) or symplectic (odd).  We have the following characterisation of the image of $\xi_\pm$ on parameters.

\begin{lemma}

For $\kappa = \pm 1$, the image of

\bes
\xi_\kappa : \Phi(U(N)) \to \Phi(G(N)) \simeq \Phi(GL(N,E))
\ees
is given by the parameters in $\Phi(GL(N,E))$ that are conjugate self-dual with parity $\kappa (-1)^N$.

\end{lemma}

\subsection{Local Arthur packets}
\label{class7}

In Section 2.5, Theorem 2.5.1, and Theorem 3.2.1 of \cite{Mo}, Mok associates a packet $\Pi_\psi$ of representations of $U(N)$ to any $\psi \in \Psi^+(U(N))$.  We recall some of the key features of this construction in the case when $\psi \in \Psi(U(N))$, which is all we shall need in this paper.  The first step is to associate to any $\psi^N \in \widetilde{\Psi}(N)$ an irreducible unitary representation of $G(N)$, denoted $\pi_{\psi^N}$.  We have the Langlands parameter $\phi_{\psi^N}$ associated to $\psi^N$, given by

\be
\label{atol}
\phi_{\psi^N}(\sigma) = \psi^N \left( \sigma, \left( \begin{array}{cc} |\sigma|^{1/2} & 0 \\ 0 & |\sigma|^{-1/2} \end{array} \right) \right), \quad \sigma \in L_F.
\ee
Let $\rho_{\psi^N}$ be the standard representation of $G(N)$ associated to $\phi_{\psi^N}$, and let $\pi_{\psi^N}$ be its Langlands quotient.  $\pi_{\psi^N}$ is an irreducible admissible conjugate self-dual representation of $G(N) \simeq GL(N,E)$, and in \cite[Section 3.2]{Mo} Mok defines a canonical extension of $\pi_{\psi^N}$ to $\widetilde{G}(N)^+$, denoted $\widetilde{\pi}_{\psi^N}$.  Mok defines a linear form on $\widetilde{\cH}(N)$ by

\begin{align*}
\widetilde{f} & \mapsto \widetilde{f}^N(\psi^N), \quad \widetilde{f} \in \widetilde{\cH}(N) \\
\widetilde{f}^N(\psi^N) & = \tr \, \widetilde{\pi}_{\psi^N}(\widetilde{f}).
\end{align*}
If $G \in \cE(U(N))$ and $\psi \in \Psi(G)$, Mok defines a linear form

\be
\label{linform}
f \mapsto f^{G}(\psi), \quad f \in \cH(G).
\ee
In the case $G = U(N)$, Mok characterizes $f^G(\psi)$ as a transfer of the linear form $\widetilde{f}^N(\xi \circ \psi)$ for $\xi = \xi_\pm$.

\begin{prop}[Theorem 3.2.1(a) of \cite{Mo}]
\label{charGL}

Let $G = U(N)$, and let $\psi \in \Psi(G)$.  For either of the embeddings $\xi_\pm$, we have

\bes
\widetilde{f}^{G}(\psi) = \widetilde{f}^N(\xi_\pm \circ \psi), \quad \widetilde{f} \in \widetilde{\cH}(N),
\ees
where $\widetilde{f}^{G}(\psi)$ denotes the evaluation of the linear form $f^{G}(\psi)$ on the transfer of $\widetilde{f}$ to $\cH(G)$ associated to $\xi_\pm$.  

\end{prop}

Proposition in fact gives a definition of $f^G(\psi)$ when $G = U(N)$, because both transfer mappings $\tcH(N) \to \cH(U(N))$ associated to $\xi_\pm$ are surjective by \cite[Prop 3.1.1(b)]{Mo}.  As a general $G \in \cE(U(N))$ is a product of the groups $U(M)$ and $G(M)$, and the definition of $f^G(\psi)$ is easy for $G(M)$ because it is a general linear group, this can be used to define $f^G(\psi)$ for all $G$.  We will only need to consider the case where $G$ is a product of two unitary groups in this paper.

We shall use the following character identities, which relate the linear forms $f^G(\psi)$ to traces of irreducible representations of $U(N)$.

\begin{prop}[Theorem 3.2.1(b) of \cite{Mo}]
\label{charid}

Let $\psi \in \Psi(U(N))$.  There exists a finite multi-set $\Pi_\psi$ whose elements are irreducible admissible representations of $U(N)$, and a mapping

\begin{align*}
\Pi_\psi & \to \widehat{\mathcal{S}}_\psi \\
\pi & \mapsto \langle \cdot, \pi \rangle
\end{align*}
with the following property.  If $s \in S_\psi$, and $(G', \psi')$ is the element of $\cE(U(N))$ corresponding to $(\psi,s)$ as in Section \ref{class61}, then we have

\bes
f^{G'}(\psi') = \sum_{\pi \in \Pi_\psi} \langle s_\psi s, \pi \rangle \textup{tr} \, \pi(f), \quad f \in \cH(U(N)).
\ees
Here we have identified $s_\psi s$ with its image in $\cS_\psi$, and $f^{G'}(\psi')$ denotes the evaluation of the linear form $f^{G'}(\psi')$ on the transfer of $f$ to $\cH(G')$.

\end{prop}

The multiset $\Pi_\psi$ is referred to as the Arthur packet associated to $\psi$.  Note that if we set $s = 1$ in Proposition \ref{charid}, then we obtain an expression for $f^{U(N)}(\psi)$ in terms of traces of the representations in $\Pi_\psi$.  Because we always have $s_\psi^2 = 1$, if we set $s = s_\psi$ we obtain

\bes
f^{G'}(\psi') = \sum_{\pi \in \Pi_\psi}  \textup{tr} \, \pi(f), \quad f \in \cH(U(N)).
\ees
We will use this to bound $\sum_{\pi \in \Pi_\psi} \dim \pi^{K}$ for various compact open subgroups $K$ of $U(N)(F)$.

\subsection{Global parameters}
\label{class8}

We now discuss the global version of the constructions of Sections \ref{class6} and \ref{class7}. For the rest of Section \ref{class} we assume that $F$ is global.  The main difficulty in adapting these constructions is that we do not have a global analogue of the Langlands group $L_F$.  However, if $L_F$ existed, its irreducible $N$-dimensional representations would correspond to cusp forms on $GL_N$.  Therefore, instead of considering representations of $L_F \times SL(2,\C)$, Mok considers formal linear combinations of products of $GL_N$ cusp forms with representations of $SL(2,\C)$, and parametrizes the spectrum of $U(N)$ using these.

For $n \ge 1$, we let $\nu(n)$ denote the unique irreducible (complex-) algebraic representation of $SL(2,\C)$ of dimension $n$.  We let $\Psi_\text{sim}(N)$ denote the set of simple global Arthur parameters, which are formal expressions $\psi^N = \mu \boxtimes \nu$ where $\mu$ is a unitary cuspidal automorphic representation of $GL(m,\A_E)$ and $\nu = \nu(n)$ for some $n$, and $N = mn$.  We let $\Psi(N)$ denote the set of global Arthur parameters, which are formal expressions

\bes
\psi^N = \psi_1^{N_1} \boxplus \dots \boxplus \psi_r^{N_r}
\ees
with $\psi_i^{N_i} \in \Psi_\text{sim}(N_i)$ and $N_1 + \dots + N_r = N$.  If $\psi^N = \mu \boxtimes \nu \in \Psi_\text{sim}(N)$, we define its conjugate dual to be $\psi^{N,*} = \mu^* \boxtimes \nu$, where $\mu^*$ is the conjugate dual representation to $\mu$, and say that $\psi^N$ is conjugate self-dual if $\psi^N = \psi^{N,*}$.  We denote the set of conjugate self-dual parameters in $\Psi_\text{sim}(N)$ by $\widetilde{\Psi}_\text{sim}(N)$.  We extend these notions to $\Psi(N)$ by defining the conjugate dual of

\bes
\psi^N = \psi_1^{N_1} \boxplus \dots \boxplus \psi_r^{N_r} \in \Psi(N)
\ees
to be

\bes
\psi^{N,*} = \psi_1^{N_1,*} \boxplus \dots \boxplus \psi_r^{N_r,*}.
\ees
We denote the set of conjugate self-dual parameters in $\Psi(N)$ by $\widetilde{\Psi}(N)$.  Note that requiring $\psi^N \in \Psi(N)$ to be conjugate self-dual is not the same as requiring that $\psi_i^{N_i} = \psi_i^{N_i,*}$ for all $i$, as we are free to rearrange the terms.  We say that

\bes
\psi^N = \psi_1^{N_1} \boxplus \dots \boxplus \psi_r^{N_r} \in \widetilde{\Psi}(N)
\ees
is elliptic if the $\psi_i^{N_i}$ are distinct and $\psi_i^{N_i} = \psi_i^{N_i,*}$ for all $i$, and denote the set of elliptic parameters by $\widetilde{\Psi}_\text{ell}(N)$.  We denote the set of generic parameters, that is those for which all the representations $\nu$ are trivial, by $\Phi(N)$, and define $\widetilde{\Phi}_*(N) = \widetilde{\Psi}_*(N) \cap \Phi(N)$.  It follows that we have chains of parameters

\begin{align*}
\widetilde{\Psi}_\text{sim}(N) & \subseteq \widetilde{\Psi}_\text{ell}(N) \subseteq \widetilde{\Psi}(N), \quad \text{and} \\
\widetilde{\Phi}_\text{sim}(N) & \subseteq \widetilde{\Phi}_\text{ell}(N) \subseteq \widetilde{\Phi}(N).
\end{align*}

To any parameter $\psi^N \in \widetilde{\Psi}(N)$, Mok \cite[Section 2.4]{Mo} associates a group $\cL_{\psi^N}$ that is an extension of $W_F$ by a complex algebraic group, and an $L$-homomorphism $\tpsi^N : \cL_{\psi^N} \times SL(2,\C) \rightarrow {}^LG(N)$.  We will not recall the definition of these objects, and give a qualitative description of them instead.  If we think of $\psi^N$ as corresponding to a hypothetical representation of $L_F \times SL(2,\C)$, $\cL_{\psi^N}$ would contain the image of this representation.  Because of this, we will use $\cL_{\psi^N}$ and $\tpsi^N$ to define what it means for $\psi^N$ to factor through the maps $\xi_\pm : {}^L U(N) \to {}^L G(N)$, and thus give a parameter for $U(N)$.

If $(U(N), \xi_\pm) \in \tcE_\text{sim}(N)$, we define $\Psi(U(N), \xi_\pm)$ to be the set of pairs $\psi = (\psi^N, \tpsi)$, where $\psi^N \in \tPsi(N)$ and

\bes
\tpsi : \cL_{\psi^N} \times SL(2,\C) \rightarrow {}^LU(N)
\ees
is an $L$-homomorphism such that $\tpsi^N = \xi_\pm \circ \tpsi$.  If $\psi = (\psi^N, \tpsi) \in \Psi(U(N), \xi_\pm)$, we set $\cL_\psi = \cL_{\psi^N}$.

\subsubsection{Parities of global parameters}

If $\phi^N \in \tPhi_\text{sim}(N)$ is associated to a conjugate self-dual cusp form $\mu$, Theorem 2.4.2 of \cite{Mo} states that there is a unique base change map $\xi_\kappa$ with $\kappa = \pm$ such that $\mu$ is the weak base change of a representation of $U(N)$ under $\xi_\kappa$.  Following Mok, we refer to $\kappa (-1)^{N-1}$ as the parity of $\phi^N$ and $\mu$.  We may extend this definition to $\psi^N = \mu \boxtimes \nu \in \widetilde{\Psi}_\text{sim}(N)$ as follows: if we assume that $\mu$ is a base change under $\xi_\delta$, we define $\kappa = \delta (-1)^{N - m - n +1}$, and define $\kappa (-1)^{N-1}$ to be the parity of $\psi^N$.  It follows from these definitions that the parity of $\mu \boxtimes \nu$ is the product of the parities of $\mu$ and $\nu$, where the parity of $\nu(n)$ is defined to be opposite to the parity of $n$ (corresponding to the fact that $\nu(n)$ is orthogonal if $n$ is odd and symplectic if $n$ is even).

This is compatible with the notion of parity discussed in Section \ref{class63}.  In particular, if $\psi \in \widetilde{\Psi}_\text{sim}(N)$ has invariant $\kappa$, then the $L$-homomorphism $\cL_\psi \times SL(2,\C) \to {}^L G(N)$ factors through $\xi_\kappa$.  In particular, if $\psi^N \in \widetilde{\Phi}_\text{sim}(N)$ then $\cL_{\psi^N} = {}^L U(N)$ and $\tpsi^N$ is the product of $\xi_\kappa$ with the trivial map on $SL(2,\C)$.  We will also see in Section \ref{class9} that if $v$ is nonsplit in $E$, the localisation $\psi_v^N : L_{F_v} \times SL(2,\C) \to {}^L G_{E_v/F_v}(N)$ of $\psi^N$ factors through the local base change map $\xi_{\pm,v}$.

\subsubsection{Square-integrable parameters}

We define $\Psi_2(U(N), \xi_\pm)$ to be the subset of $\Psi(U(N), \xi_\pm)$ for which $\psi^N \in \tPsi_\text{ell}(N)$.  This is known as the set of square-integrable parameters of $U(N)$ with respect to $\xi_\pm$, because these are the parameters that give the discrete automorphic spectrum of $U(N)$.  In concrete terms, a parameter $\psi^N \in \tPsi_\text{ell}(N)$ can be extended to $\psi = (\psi^N, \tpsi) \in \Psi_2(U(N), \xi_\kappa)$ if and only if $\psi^N = \psi_1^{N_1} \boxplus \ldots \boxplus \psi_l^{N_l}$ with the parameters $\psi_i^{N_i} \in \widetilde{\Psi}_\text{sim}(N_i)$ all having parity $\kappa (-1)^{N-1}$.  More concretely, if $\psi_i^{N_i} = \phi_i \boxtimes \nu(n_i)$ with $\phi_i \in \widetilde{\Phi}_\text{sim}(m_i)$, we require that $\delta_i (-1)^{m_i + n_i} = \kappa (-1)^N$ for all $i$, where $\delta_i$ is such that the cusp form $\mu_i$ associated to $\phi_i$ is a weak base change from $U(m_i)$ under $\xi_{\delta_i}$.

\subsection{Localisation of parameters}
\label{class9}

Having introduced global and local versions of our parameters, we now discuss the localisation maps taking the former to the latter.  We let $v$ be a place of $F$, and let $E_v = E \otimes_F F_v$, $U(N)_v = U_{E_v/F_v}$, and $G(N)_v = G_{E_v/F_v}(N)$.

We first assume that $v$ does not split in $E$.    Consider a simple generic parameter $\phi^N \in \Phi_\text{sim}(N)$.  As $\phi^N$ corresponds to a cusp form $\mu$ on $GL(N,\A_E)$, we may consider the local factor $\mu_v$, which is an irreducible unitary representation of $GL(N,E_v)$.  By the local Langlands correspondence for $GL(N)$ by Harris-Taylor \cite{HT} and Henniart \cite{H}, and the isomorphism $\Phi(GL(N,E_v)) \simeq \Phi(G(N)_v)$ of Section \ref{class62}, $\mu_v$ corresponds to a local Langlands parameter $\phi_v^N \in \Phi_v(N) := \Phi(G(N)_v)$.  This gives the localisation map from $\Phi_\text{sim}(N)$ to $\Phi_v(N)$, which takes $\widetilde{\Phi}_\text{sim}(N)$ to $\widetilde{\Phi}_v(N)$.  This may be naturally extended to a map $\psi^N \mapsto \psi^N_v$ from $\Psi(N)$ to $\Psi^+_v(N)$ that takes $\widetilde{\Psi}(N)$ to $\widetilde{\Psi}^+_v(N)$.

Now consider a parameter $\psi = (\psi^N, \tpsi) \in \Psi(U(N), \xi_\kappa)$.  By \cite[Corollary 2.4.11]{Mo}, the localisation $\psi_v^N$ factors through the embedding $\xi_{\kappa,v} : {}^L U(N)_v \to {}^L G(N)_v$.  This allows us to define $\psi_v \in \Psi(U(N)_v)$ by requiring that $\xi_{\kappa,v} \circ \psi_v = \psi_v^N$.

We now assume that $v$ splits in $E$, and write $v = w \overline{w}$.  As in Section \ref{class2} we have isomorphisms $\iota_w : U(N)_v \to GL(N,E_w)$ and $\iota_{\overline{w}} : U(N)_v \to GL(N,E_{\overline{w}})$ corresponding to the projections of $E_v$ to $E_w$ and $E_{\overline{w}}$.  If $\psi = (\psi^N, \tpsi) \in \Psi(U(N), \xi_\kappa)$, we may think of the localisations $\psi_w^N$ and $\psi_{\overline{w}}^N$ as elements of $\Psi^+(GL(N,E_w))$ and $\Psi^+(GL(N,E_{\overline{w}}))$.  When $\psi_w^N \in \Psi(GL(N,E_w))$, we may define $\pi_{\psi_w^N}$ to be the representation associated to $\phi_{\psi_w^N}$ by local Langlands, where $\phi_{\psi_w^N}$ is as in (\ref{atol}).  The definition of $\pi_{\psi_w^N}$ for $\psi_w^N \in \Psi^+(GL(N,E_w))$ is given in \cite[Section 2.4]{Mo}, and will not be needed in this paper because the $GL_N$ cusp forms we consider are known to satisfy the Ramanujan conjectures.

The conjugate self-duality of $\psi^N$ implies that $\pi_{\psi_w^N} = (\pi_{\psi_{\overline{w}}^N})^\vee$, and $\iota_w \circ \iota_{\overline{w}}^{-1}$ is the automorphism $g \mapsto J {}^t g^{-1} J^{-1}$ of Section \ref{class2} (under the identification $E_w = E_{\overline{w}} = F_v$).  Therefore the pullback of $\pi_{\psi_w^N}$ via $\iota_w$ is isomorphic to the pullback of $\pi_{\psi_{\overline{w}}^N}$ via $\iota_{\overline{w}}$.  We denote this representation of $U(N)_v$ by $\pi_{\psi_v}$.  We define $\psi_v \in \Psi^+(U(N)_v)$ to be the parameter obtained by composing $\psi_w^N : L_{F_v} \times SL(2,\C) \simeq L_{E_w} \times SL(2,\C) \to {}^L GL(N,E_w)$ with the isomorphism ${}^L \iota_w : {}^L GL(N,E_w) \to {}^L U(N)_v$ induced by $\iota_w$.  We define $\Pi_{\psi_v} = \{ \pi_{\psi_v} \}$ to be the local Arthur packet associated to $\psi_v$.

\subsection{The global classification}
\label{class10}

We may now state the global classification theorem.  For any $\psi$ in the set of global parameters $\Psi_2(U(N),\xi_\pm)$, we have the localisations $\psi_v$ and the local Arthur packets $\Pi_{\psi_v}$ associated to $\psi_v$ in Sections \ref{class7} and \ref{class9}.  We define the global Arthur packet $\Pi_\psi$ to be the restricted direct product of the $\Pi_{\psi_v}$, in the sense that it contains those $\otimes_v \pi_v \in \otimes_v \Pi_{\psi_v}$ such that the (global analogue of the) character $\langle \cdot, \pi_v \rangle$ is trivial for almost all $v$.  We will write $\Pi_\psi = \otimes_v \Pi_{\psi_v}$ by slight abuse of notation.  In \cite[Section 2.5]{Mo}, Mok defines a subset $\Pi_\psi(\epsilon_\psi) \subset \Pi_\psi$ in terms of symplectic root numbers and the pairings in Proposition \ref{charid}, which we do not need to make explicit.  The classification is as follows.

\begin{theorem}
\label{UNclass}

For $\kappa = \pm 1$, we have a $U(N)(\A)$-module decomposition of the discrete automorphic spectrum of $U(N)$:

\bes
L^2_\textup{disc}(U(N)(F) \backslash U(N)(\A)) = \sum_{\psi \in \Psi_2(U(N),\xi_\kappa)} \sum_{\pi \in \Pi_\psi(\epsilon_\psi)} \pi.
\ees

\end{theorem}

Mok's proof of Theorem \ref{UNclass} builds on work by many authors, notably Arthur, who classified the discrete spectrum of quasi-split symplectic and orthogonal groups in \cite{A}, and Moeglin and Waldspurger, who proved the stabilization of the twisted trace formula.  Theorem \ref{UNclass} is being extended to general forms of unitary groups by Kaletha, Minguez, Shin, and White in \cite{KMSW} and its projected sequels.  In joint work with Shin, we hope to show that this extension of Theorem \ref{UNclass} implies strong (and conjecturally sharp) upper bounds for cohomology growth on arithmetic manifolds associated to $U(n,1)$ for any $n$.

\section{Application of the Global Classification}

In this section, we rephrase Theorem \ref{simplemainthm} in terms of Arthur packets by applying the results of Section \ref{class} to the manifolds $Y(n)$.

\subsection{Notation}

Let $E$ be an imaginary quadratic field with ring of integers $\cO$.  We apply the notation of Section \ref{class} to the extension $E/\Q$.  We denote places of $\Q$ and $E$ by $v$ and $w$ respectively.  We recall the character $\chi$ of $E^\times \backslash \A_E^\times$ whose restriction to $\A^\times$ is the character associated to $E/\Q$ by class field theory.  We let $S_f$ be a finite set of finite places of $\Q$ that contains all finite places at which $E$ is ramified, and all finite places that are divisible by a place of $E$ at which $\chi$ is ramified.

If $G$ is an algebraic group over $\Q$ or $\Q_v$, we denote $G(\Q_v)$ by $G_v$, and likewise for groups over $E$.  For any $N \ge 1$ we let $\widetilde{G}(N)_v = G(N)_v \rtimes \theta$, and $\tcH_v(N) = C^\infty_0( \widetilde{G}(N)_v)$.  We fix Haar measures on $U(N)_v$ and $\widetilde{G}(N)_v$ for all $N \ge 1$ and all $v$, subject to the condition that these measures assign volume 1 to a hyperspecial maximal compact when $v$ is finite and the groups are unramified.  All traces and twisted traces will be defined with respect to these measures.

We shall identify the infinitesimal character of an irreducible admissible representation of $U(N)_\infty$ and $GL(N,\C)$ with a point in $\C^N / S_N$ and $(\C^N / S_N) \times (\C^N / S_N)$ respectively, where $S_N$ is the symmetric group.

We choose a compact open subgroup $K = \prod_p K_p \subset U(4)(\A_f)$, subject to the condition that $K_p = U(4)(\Z_p)$ for $p \notin S_f$.  For any $n \ge 1$ that is relatively prime to $S_f$, we define $K_p(n)$ to be the subgroup of $K_p$ consisting of elements congruent to $1$ modulo $n$ when $p \notin S_f$, and $K_p(n) = K_p$ otherwise, and define $K(n) = \prod_p K_p(n)$.

We let $K_\infty$ be the standard maximal compact subgroup of $U(4)_\infty$.  For any $n \ge 1$ that is relatively prime to $S_f$, we define $Y(n) = U(4)(\Q) \backslash U(4)(\A) /  K_\infty K(n)$.  For any $0 \le i \le 8$, we let $h^i_{(2)}(Y(n))$ denote the dimension of the space of square integrable harmonic $i$-forms on $Y(n)$.

\subsection{Reduction of Theorem \ref{simplemainthm} to Arthur packets}

The precise form of Theorem \ref{simplemainthm} we shall prove is the following.

\begin{theorem}
\label{mainthm}

If $i = 2, 3$, and $n$ is relatively prime to $S_f$ and divisible only by primes that split in $E$, we have $h^i_{(2)}(Y(n)) \ll n^9$.

\end{theorem}

The implied constant depends only on $K$, and we shall ignore the dependence of implied constants on $K$ for the rest of the paper.  By considering the action of the center on the connected components of $Y(n)$, Theorem \ref{mainthm} implies that the connected component $Y^0(n)$ of the identity satisfies $h^i_{(2)}(Y^0(n)) \ll_\epsilon n^{8 + \epsilon}$.  This implies Theorem \ref{simplemainthm} when combined with the asymptotic $\text{Vol}(Y^0(n)) = n^{15 + o(1)}$.

We shall only prove Theorem \ref{mainthm} in the case $i = 3$, as the case $i = 2$ is identical.  We begin by applying the extension of Matsushima's formula to noncompact quotients \cite[Prop 5.6]{BG}, which gives

\be
\label{matsushima}
h^3_{(2)}(Y(n)) = \sum_{\pi \in L^2_\text{disc}(U(4)(\Q) \backslash U(4)(\A))} h^3(\g,K;\pi_\infty) \dim \pi_f^{K(n)}.
\ee
If we combine this with Theorem \ref{UNclass}, we obtain

\be
\label{U4class}
h^3_{(2)}(Y(n)) \le \sum_{\psi \in \Psi_2(U(4), \xi_+) } \sum_{\pi \in \Pi_\psi} h^3(\g,K;\pi_\infty) \dim \pi_f^{K(n)}.
\ee
It follows from the proof of the Adams-Johnson conjectures in \cite{AMR}, or Proposition 13.4 of \cite{BMM}, that if $\pi \in \Pi_\psi$ satisfies $h^3(\g,K;\pi_\infty) \neq 0$, then $\psi$ is not generic.  It follows that $\psi^N$ must be of one of the following types.

\begin{enumerate}[(a)]

\item
\label{para}
$\nu(2) \boxtimes \phi^N_1 \boxplus \phi^N_2$, $\phi^N_i \in \widetilde{\Phi}_\text{ell}(i)$.

\item
\label{parb}
$\nu(2) \boxtimes \phi^N$, $\phi^N \in \widetilde{\Phi}_\text{ell}(2)$.

\item
\label{parc}
$\nu(3) \boxtimes \phi^N_1 \boxplus \phi^N_2$, $\phi^N_i \in \widetilde{\Phi}(1)$.

\item
\label{pard}
$\nu(4) \boxtimes \phi^N$,  $\phi^N \in \widetilde{\Phi}(1)$.

\end{enumerate}

We bound the contribution of parameters of types (\ref{para}) and (\ref{parb}) in $\mathsection$\ref{secpara} and $\mathsection$\ref{secparb} respectively.  It follows from the description of the packets $\Pi_\psi$ at split places that all representations contained in packets of type (\ref{pard}) must be characters, and these make a contribution of $\ll_\epsilon n^{1 + \epsilon}$ to $h^3_{(2)}(Y(n))$.  We shall also omit the case of parameters of type (\ref{parc}); it may be proven that they make a contribution of $\ll_\epsilon n^{5 + \epsilon}$ using the same methods as in $\mathsection$\ref{secparb}.

\section{The case $\psi^N = \nu(2) \boxtimes \phi^N_1 \boxplus \phi^N_2$}
\label{secpara}

Let $h^3_{(2)}(Y(n))^\star$ denote the contribution to $h^3_{(2)}(Y(n))$ from parameters of the form $\nu(2) \boxtimes \phi_1^N \boxplus \phi_2^N$, which by (\ref{U4class}) satisfies

\be
\label{U4class1}
h^3_{(2)}(Y(n))^\star \le \sum_{ \substack{ \psi \in \Psi_2(U(4), \xi_+) \\ \psi^N = \nu(2) \boxtimes \phi^N_1 \boxplus \phi^N_2 }} \sum_{\pi \in \Pi_\psi} h^3(\g,K;\pi_\infty) \dim \pi_f^{K(n)}.
\ee
We assume that the sum is restricted to those $\phi_2^N$ lying in $\tPhi_\text{sim}(2)$ until the end of Section \ref{secparasum}, and describe how to treat composite $\phi_2^N$ in $\mathsection$\ref{secparacomposite}.  We note that $\psi \in \Psi_2(U(4), \xi_+)$ implies that $\phi_1^N$ and $\phi_2^N$ must be even and odd respectively.  The main result of this section is the following.

\begin{prop}
\label{paraprop}

We have the bound $h^3_{(2)}(Y(n))^\star \ll n^9$.

\end{prop}

For $i = 1, 2$, we let $K_i = \prod_p K_{i,p}$ be a compact open subgroup of $U(i)(\A_f)$ such that $K_{i,p} = U(i)(\Z_p)$ for all $p \notin S_f$, and let $\widetilde{K}_i = \prod_w \widetilde{K}_{i,w}$ be a compact open subgroup of $GL(i,\A_{E,f})$ such that $\widetilde{K}_{i,w} = GL(i,\cO_w)$ for all $w|p \notin S_f$.  We define $\widetilde{K} \subset GL(4,\A_{E,f})$ in a similar way.  The groups $K_{2,p}$ and $\widetilde{K}_{1,w}$ for $w|p \in S_f$ will be specified in the proof of Proposition \ref{paralocal}, and the groups $K_{1,p}$, $\widetilde{K}_{2,w}$, and $\widetilde{K}_w$ for $w|p \in S_f$ may be chosen arbitrarily.  We define congruence subgroups $K_*(n)$ of these groups for $n$ relatively prime to $S_f$ in the usual way, and recall that $n$ will only be divisible by primes that split in $E$.

We let $\widetilde{P}$ be the standard parabolic subgroup of $GL(4,E)$ with Levi $\widetilde{L} = GL(2,E) \times GL(2,E)$, and let $P$ be the corresponding standard parabolic subgroup of $U(4)$.

\subsection{Controlling a single parameter}

We first bound the contribution from a single Arthur parameter to $h^3_{(2)}(Y(n))^\star$.  We therefore fix $\phi^N_i \in \tPhi_\text{sim}(i)$ for $i = 1, 2$ with $\phi_1^N$ even and $\phi_2^N$ odd, and let $\psi \in \Psi(U(4), \xi_+)$ be the unique parameter with $\psi^N = \nu(2) \boxtimes \phi^N_1 \boxplus \phi^N_2$.  We let $\phi_i^N$ correspond to a conjugate self-dual cuspidal automorphic representation $\mu_i$ of $GL(i,\A_E)$.  We assume that $\mu_i$ are tempered at all places.  This assumption is not necessary, but simplifies the proof of Proposition \ref{paralocal} and will be proven to hold for all parameters that contribute to cohomology.

We define $\psi^N_1 = \nu(2) \boxtimes \phi_1^N$ and $\psi^N_2 = \phi_2^N$, and for $i = 1,2$ we let $\psi_i \in \Psi(U(2), \xi_+)$ be the corresponding unitary parameters.  We shall prove the following bound for the finite part of the contribution of $\Pi_\psi$ to $h^3_{(2)}(Y(n))^\star$.

\begin{prop}
\label{paralocal}

There is a choice of $\widetilde{K}_{1,w}$ for $w | p$, $p \in S_f$, and $K_{2,p}$ for $p \in S_f$, depending only on $K$, such that

\bes
\sum_{\pi_f \in \Pi_{\psi,f}} \dim \pi_f^{K(n)} \ll [ K : (K \cap P(\A_f)) K(n)] \dim \mu_1^{\widetilde{K}_1(n)} \sum_{\pi_f' \in \Pi_{\psi_2,f} } \dim \pi_f'^{K_2(n)},
\ees
where $\Pi_{\psi,f} = \otimes_p \Pi_{\psi_p}$ is the finite part of $\Pi_\psi$, and likewise for $\Pi_{\psi_2}$.

\end{prop}

The proposition will follow from the factorization of $\Pi_{\psi,f}$, and the series of lemmas below.

\begin{lemma}
\label{paralocalnonsplit}

Let $p \notin S_f$ be nonsplit in $E$, and let $w | p$.  We have

\bes
\sum_{\pi_p \in \Pi_{\psi_p}} \dim \pi_p^{K_p} = \dim \mu_{1,w}^{\widetilde{K}_{1,w}} \sum_{\pi_p' \in \Pi_{\psi_{2,p}} } \dim \pi_p'^{K_{2,p}}.
\ees

\end{lemma}

\begin{proof}

We have

\bes
\sum_{\pi_p \in \Pi_{\psi_p}} \dim \pi_p^{K_p} = \sum_{\pi_p \in \Pi_{\psi_p}} \tr( \pi_p(1_{K_p}) ),
\ees
and we may manipulate the right hand side using the local character identities of Propositions \ref{charGL} and \ref{charid}.  Let $(G', \xi') \in \cE_\text{ell}(U(4)_p)$ be the unique endoscopic datum with $G' = U(2)_p \times U(2)_p$, and let $\psi'_p = \psi_{1,p} \times \psi_{2,p} \in \Psi(G')$.  It may be seen that $(G',\psi'_p)$ is the pair associated to $(\psi_p,s_\psi)$ by the correspondence of $\mathsection$\ref{class61}.  We recall the distribution $f \mapsto f^{G'}(\psi_p')$ on $\cH(G')$ associated to $\psi_p'$ in (\ref{linform}).  Applying Proposition \ref{charid} with $s = s_{\psi_p}$, and the fundamental lemma for the group $G' \in \cE(U(4)_p)$, gives

\bes
\sum_{\pi_p \in \Pi_{\psi_p}} \tr( \pi_p(1_{K_p}) ) = (1_{K_{2,p}} \times 1_{K_{2,p}})^{G'}(\psi_p').
\ees

Because $\psi_p' = \psi_{1,p} \times \psi_{2,p}$, the factorisation property of the linear form $f^{G'}(\psi_p')$ allows us to write this as

\bes
\sum_{\pi_p \in \Pi_{\psi_p}} \tr( \pi_p(1_{K_p}) ) = 1_{K_{2,p}}^{U(2)}(\psi_{1,p}) 1_{K_{2,p}}^{U(2)}(\psi_{2,p}),
\ees
where $f \mapsto f^{U(2)}(\psi_{i,p})$ are the distributions on $\cH(U(2)_p)$ associated to $\psi_{i,p}$.  Because $s_{\psi_{i,p}} = e$ for $i = 1,2$, we may express $1_{K_{2,p}}^{U(2)}(\psi_{i,p})$ in terms of traces of representations by applying Proposition \ref{charid} with $s = e$, which gives

\be
\label{positive}
1_{K_{2,p}}^{U(2)}(\psi_{i,p}) = \sum_{\pi_p' \in \Pi_{\psi_{i,p}} } \tr( \pi_p'(1_{K_{2,p}}) ) = \sum_{\pi_p' \in \Pi_{\psi_{i,p}} } \dim \pi_p'^{K_{2,p}}.
\ee
This gives the required expression for $1_{K_{2,p}}^{U(2)}(\psi_{2,p})$.

We evaluate $1_{K_{2,p}}^{U(2)}(\psi_{1,p})$ by applying Proposition \ref{charGL} with the embedding $\xi$ chosen to be $\xi_+ : {}^LU(2)_p \rightarrow {}^LG(2)_p$.  If we restrict the map

\bes
\xi_+ \circ \psi_{1,p} : L_{\Q_p} \times SL(2,\C) \rightarrow {}^LG(2)_p
\ees
to $L_{E_w} \times SL(2,\C)$, it is equivalent to

\begin{align*}
\xi_+ \circ \psi_{1,p} : L_{E_w} \times SL(2,\C) & \rightarrow GL(2,\C) \\
\sigma \times A & \mapsto \phi_{1,w}^N(\sigma) A.
\end{align*}
It follows that the representation of $G(2)_p \simeq GL(2,E_w)$ associated to $\xi_+ \circ \psi_{1,p}$ is equal to $\mu_{1,w} \circ \det$.  We denote the canonical extension of this representation to $\widetilde{G}^+(2)_p$ by $\widetilde{\pi}_1$.  If we identify $\widetilde{K}_{2,w}$ with a subgroup of $G(2)_p$, the twisted fundamental lemma implies that we may take $\widetilde{f} = 1_{\widetilde{K}_{2,w} \rtimes \theta} \in \tcH_p(2)$ in Proposition \ref{charGL} to obtain

\bes
1_{K_{2,p}}^{U(2)}(\psi_{1,p}) = \tr( \widetilde{\pi}_1( 1_{\widetilde{K}_{2,w} \rtimes \theta} ) ). 
\ees
Because $\theta^2 = 1$, we have

\bes
\tr( \widetilde{\pi}_1( 1_{\widetilde{K}_{2,w} \rtimes \theta} ) ) = \pm \dim \widetilde{\pi}_1^{\widetilde{K}_{2,w}} = \pm \dim \mu_{1,w}^{\widetilde{K}_{1,w}}.
\ees
Applying equation (\ref{positive}) with $i = 1$ implies that $1_{K_{2,p}}^{U(2)}(\psi_{1,p}) \ge 0$, which means that we must take the positive sign.  This completes the proof.

\end{proof}

\begin{lemma}
\label{unramsplit}

Let $p \notin S_f$ be split in $E$, and let $w|p$.  Let $\Pi_{\psi_p} = \{ \pi_p \}$, and $\Pi_{\psi_{2,p}} = \{ \pi_p' \}$.  We have

\bes
\dim \pi_p^{K_p(n)} = [K_p : (K_p \cap P_p) K_p(n)] \dim \mu_{1,w}^{\widetilde{K}_{1,w}(n)} \dim \pi_p'^{K_{2,p}(n)}.
\ees

\end{lemma}

\begin{proof}

Under the identification $U(4)_p \simeq GL(4,E_w)$, the discussion of Section \ref{class9} implies that $\pi_p$ is isomorphic to the representation induced from the representation $(\mu_{1,w} \circ \det) \otimes \mu_{2,w}$ of $\widetilde{P}_w$.  The restriction of $\pi_p$ to $K_p$ is isomorphic to the induction of $(\mu_{1,w} \circ \det) \otimes \mu_{2,w}$ from $\widetilde{P}_w \cap \widetilde{K}_w$ to $\widetilde{K}_w$.  Because $\widetilde{K}_w(n) \cap \widetilde{L}_w = \widetilde{K}_{2,w}(n) \times \widetilde{K}_{2,w}(n)$, and $\dim (\mu_{1,w} \circ \det)^{\widetilde{K}_{2,w}(n)} = \dim \mu_{1,w}^{\widetilde{K}_{1,w}(n)}$, we have

\bes
\dim \pi_p^{K_p(n)} = [\widetilde{K}_w : (\widetilde{K}_w \cap \widetilde{P}_w) \widetilde{K}_w(n)] \dim \mu_{1,w}^{\widetilde{K}_{1,w}(n)} \dim \pi_p'^{K_{2,p}(n)}
\ees
which is equivalent to the lemma.

\end{proof}

\begin{lemma}

Let $p \in S_f$, and let $w|p$.  There is a choice of $\widetilde{K}_{1,w}$ and $K_{2,p}$, depending only on $K_p$, such that

\bes
\sum_{\pi_p \in \Pi_{\psi_p}} \dim \pi_p^{K_p} \ll \dim \mu_{1,w}^{\widetilde{K}_{1,w}} \sum_{\pi_p' \in \Pi_{\psi_{2,p}}} \dim \pi_p'^{K_{2,p}}.
\ees

\end{lemma}

\begin{proof}

If $p$ is split, this follows from the explicit description of $\Pi_{\psi_p}$ as in Lemma \ref{unramsplit}.  Assume that $p$ is nonsplit, and continue to use the notation of Lemma \ref{paralocalnonsplit}.  Let $\widetilde{1}_{K_p} \in \cH(G')$ be a transfer of $1_{K_p}$ to $G'$.  Reasoning as in the proof of Lemma \ref{paralocalnonsplit} gives

\bes
\sum_{\pi_p \in \Pi_{\psi_p}} \dim \pi_p^{K_p} = \text{vol}(K_p)^{-1} \widetilde{1}_{K_p}^{G'}(\psi_p'),
\ees
where $\text{vol}(K_p)$ denotes the volume of $K_p$ with respect to our chosen Haar measure on $U(4)_p$.  We may write $\widetilde{1}_{K_p} = \sum f_{i,1} \times f_{i,2}$ for $f_{i,j} \in \cH(U(2)_p)$, and the factorisation property of $f^{G'}(\psi_p')$ gives

\begin{align*}
\widetilde{1}_{K_p}^{G'}(\psi_p') & = \sum_i (f_{i,1} \times f_{i,2})^{G'}(\psi_p') \\
& = \sum_i f_{i,1}^{U(2)}(\psi_{1,p}) f_{i,2}^{U(2)}(\psi_{2,p}).
\end{align*}
Applying Proposition \ref{charid} with $s = e$ gives

\begin{align*}
f_{i,2}^{U(2)}(\psi_{2,p}) & = \sum_{\pi_p' \in \Pi_{\psi_{2,p}} } \tr( \pi_p'(f_{i,2}) ) \\
& \le C(f_{i,2}) \sum_{\pi_p' \in \Pi_{\psi_{2,p}} } \dim \pi_p'^{K_{2,p}}
\end{align*}
if $K_{2,p}$ is chosen so that $f_{i,2}$ is bi-invariant under $K_{2,p}$ for all $i$.  Likewise, applying Proposition \ref{charGL} and the definition of $\widetilde{\pi}_{\psi_{1,p}^N}$ shows that $f_{i,1}^{U(2)}(\psi_{1,p}) \le C(f_{i,1}) \dim \mu_{1,w}^{\widetilde{K}_{1,w}}$ if $\widetilde{K}_{1,w}$ is chosen sufficiently small depending on $f_{i,1}$.  As the collection of functions $f_{i,j}$ depended only on $K_p$, so do $\widetilde{K}_{1,w}$ and $K_{2,p}$, and the constant factors.

\end{proof}

\subsection{Summing over parameters}
\label{secparasum}

We now use Proposition \ref{paralocal} to control the contribution to $h^3_{(2)}(Y(n))^\star$ from all $\psi$.

\begin{lemma}
\label{archlocal}

Let $\psi \in \Psi(U(4), \xi_+)$, and suppose that $\psi^N = \nu(2) \boxtimes \phi_1^N \boxplus \phi_2^N$ with $\phi_i^N \in \tPhi_\textup{sim}(i)$.  If $\pi \in \Pi_{\psi_\infty}$ satisfies $H^*(\g,K;\pi) \neq 0$, then we have

\begin{align*}
\phi^N_{1,\infty} : z & \mapsto (z/\overline{z})^{\alpha'} \\
\phi^N_{2,\infty} : z & \mapsto \left( \begin{array}{cc} (z/\overline{z})^{\alpha_1} & \\ & (z/\overline{z})^{\alpha_2} \end{array} \right)
\end{align*}
with $\alpha' \in \{ 1, 0, -1 \}$, $\alpha_i \in \{ 3/2, 1/2, -1/2, -3/2 \}$, and $\alpha_1 \neq \alpha_2$.

\end{lemma}

\begin{proof}

We write

\begin{align*}
\phi^N_{1,\infty} : z & \mapsto z^{\alpha'} \overline{z}^{\beta'} \\
\phi^N_{2,\infty} : z & \mapsto \left( \begin{array}{cc} z^{\alpha_1} \overline{z}^{\beta_1} & \\ & z^{\alpha_2} \overline{z}^{\beta_2} \end{array} \right)
\end{align*}
with $\alpha' - \beta', \alpha_i - \beta_i \in \Z$.  If we let $\phi_{\psi_\infty}$ be the Langlands parameter associated to $\psi_\infty$ as in (\ref{atol}), any $\pi \in \Pi_{\psi_\infty}$ has the same infinitesimal character as the representations in the $L$-packet of $\phi_{\psi_\infty}$, which is $(\alpha'+1/2, \alpha' - 1/2, \alpha_1, \alpha_2) \in \C^4 / S_4$ (see for instance \cite[Prop 7.4]{Vo}).  If $\pi$ is to have cohomology it must have the same infinitesimal character as the trivial representation, so that $\{\alpha'+1/2, \alpha' - 1/2, \alpha_1, \alpha_2 \} = \{ 3/2, 1/2, -1/2, -3/2\}$.  This implies that $\alpha' \in \{ 1, 0, -1 \}$ and $\alpha_i \in \{ 3/2, 1/2, -1/2, -3/2\}$ with $\alpha_1 \neq \alpha_2$.   Because $\mu_1$ is a character we have $\alpha' = -\beta'$, and because $\mu_2$ is a cusp form on $GL(2,E)$ we have $|\alpha_i + \beta_i| < 1/2$ so that $\alpha_i = -\beta_i$.  This completes the proof.

\end{proof}

For $i = 1,2$, we define $\Phi_\text{rel}(i) \subset \tPhi_\text{sim}(i)$ to be the set of parameters $\phi_i^N$ such that $\phi^N_{i,\infty}$ satisfies the relevant constraints of Lemma \ref{archlocal}.  If $\phi_2^N \in \Phi_\text{rel}(2)$ is associated to a cuspidal representation $\mu$, it follows that $\mu$ is regular algebraic, conjugate self-dual, and cuspidal, and hence tempered at all places by Theorem 1.2 of \cite{Ca}.

 Lemma \ref{archlocal} and equation (\ref{U4class1}) imply that

\begin{align*}
h^3_{(2)}(Y(n))^\star & \ll \sum_{ \substack{  \psi^N = \nu(2) \boxtimes \phi^N_1 \boxplus \phi^N_2 \\ \phi_i^N \in \Phi_\text{rel}(i) }}  \sum_{\pi \in \Pi_\psi} \dim \pi_f^{K(n)} \\
& = \sum_{ \substack{  \psi^N = \nu(2) \boxtimes \phi^N_1 \boxplus \phi^N_2 \\ \phi_i^N \in \Phi_\text{rel}(i) }} \# ( \Pi_{\psi_\infty}) \sum_{\pi_f \in \Pi_{\psi,f}} \dim \pi_f^{K(n)}.
\end{align*}
We may ignore the factor $\# (\Pi_{\psi_{\infty}})$ because there are only finitely many possibilities for $\psi_{\infty}$.  Applying Proposition \ref{paralocal} to the right hand side gives

\bes
h^3_{(2)}(Y(n))^\star \ll [ K : (K \cap P(\A_f)) K(n)] \sum_{  \phi_1^N \in \Phi_\text{rel}(1) } \dim \mu_1^{\widetilde{K}_1(n)} \sum_{  \phi_2^N \in \Phi_\text{rel}(2) }  \sum_{\pi'_f \in \Pi_{\psi_2,f}} \dim \pi_f'^{K_2(n)},
\ees
where $\mu_1$ is the automorphic character associated to $\phi_1^N$.  We may enlarge the sum from $\Pi_{\psi_2,f}$ to $\Pi_{\psi_2}$, which gives

\be
\label{asum2}
h^3_{(2)}(Y(n))^\star \ll [ K : (K \cap P(\A_f)) K(n)] \sum_{  \phi_1^N \in \Phi_\text{rel}(1) } \dim \mu_1^{\widetilde{K}_1(n)} \sum_{  \phi_2^N \in \Phi_\text{rel}(2) }  \sum_{\pi' \in \Pi_{\psi_2}} \dim \pi_f'^{K_2(n)}.
\ee

Lemma \ref{archlocal} implies that there are only three possibilities for $\mu_{1,\infty}$, and therefore

\be
\label{asum3}
\sum_{  \phi_1^N \in \Phi_\text{rel}(1) } \dim \mu_1^{\widetilde{K}_1(n)} \ll [K_1 : K_1(n)].
\ee
There is a finite set $\Xi_\infty$ of representations of $U(2)_\infty$ such that if $\phi_2^N \in \Phi_\text{rel}(2)$ and $\pi' \in \Pi_{\psi_2}$, then $\pi'_\infty \in \Xi_\infty$.  Moreover, because $\psi_2$ is a simple generic parameter, we have $\Pi_\psi(\epsilon_\psi) = \Pi_\psi$ and so every $\pi' \in \Pi_{\psi_2}$ occurs in $L^2_\text{disc}(U(2)(\Q) \backslash U(2)(\A))$ with multiplicity one.  We define $X(n) = U(2)(\Q) \backslash U(2)(\A) / K_2(n)$, and let $m(\pi_\infty, X(n) )$ denote the multiplicity with which a representation $\pi_\infty$ occurs in $L^2_\text{disc}(X(n))$.  We have

\begin{align}
\notag
\sum_{  \phi_2^N \in \Phi_\text{rel}(2) } \sum_{\pi' \in \Pi_{\psi_2}} \dim \pi_f'^{K_2(n)} & \le \sum_{ \substack{ \pi' \in  L^2_\text{disc}(U(2)(\Q) \backslash U(2)(\A)) \\ \pi'_\infty \in \Xi_\infty} } \dim \pi_f'^{K_2(n)} \\
\notag
& = \sum_{\pi_\infty \in \Xi_\infty} m(\pi_\infty, X(n) ) \\
\label{asum4}
& \ll [ K_2 : K_2(n) ].
\end{align}
Combining (\ref{asum2})--(\ref{asum4}) gives

\bes
h^3_{(2)}(Y(n))^\star \ll [ K : (K \cap P(\A_f)) K(n)] [ K_2 : K_2(n) ] [K_1 : K_1(n)].
\ees
Applying the formula for the order of $GL(N)$ over a finite field completes the proof.

\subsection{The case of $\phi_2^N$ composite}
\label{secparacomposite}

We now briefly explain how to bound the contribution to $h^3_{(2)}(Y(n))^\star$ from parameters with $\phi_2^N = \phi_{21}^N \boxplus \phi_{22}^N$, where $\phi_{2i}^N \in \tPhi(1)$.  We let $\phi_{2i}^N$ correspond to a conjugate self-dual character $\mu_{2i}$ on $GL(1,\A_E)$.  Let $P_2$ be the standard Borel subgroup of $U(2)$.  We may prove the following analogue of Proposition \ref{paralocal}.

\begin{prop}

There is a choice of $\widetilde{K}_{1,w}$ for $w | p$, $p \in S_f$, depending only on $K$, such that

\begin{multline}
\label{paracomp}
\sum_{\pi_f \in \Pi_{\psi,f}} \dim \pi_f^{K(n)} \ll [ K : (K \cap P(\A_f)) K(n)] [ K_2 : (K_2 \cap P_2(\A_f)) K_2(n) ] \\
\dim \mu_1^{\widetilde{K}_1(n)} \dim \mu_{21}^{\widetilde{K}_1(n)} \dim \mu_{22}^{\widetilde{K}_1(n)}.
\end{multline}

\end{prop}

The proof follows the same lines, by using the explicit description of $\pi_{\psi_p}$ when $p$ is split and the character identities of Propositions \ref{charGL} and \ref{charid} when $p$ is inert.  There are $\ll n^3$ choices for the three characters, and the coset factors in (\ref{paracomp}) make a contribution of $\ll_\epsilon n^{5 + \epsilon}$.  Therefore the contribution to cohomology of parameters of this type is bounded by $\ll_\epsilon n^{8 + \epsilon}$ as required.

\section{The case $\psi^N = \nu(2) \boxtimes \phi^N$}
\label{secparb}

We now define $h^3_{(2)}(Y(n))^\star$ to be the contribution to $h^3_{(2)}(Y(n))$ from parameters of the form $\nu(2) \boxtimes \phi^N$.  As in Section \ref{secpara}, we assume that $\phi^N \in \tPhi_\text{sim}(2)$ until the end of Section \ref{secparbsum}, and describe how to treat composite $\phi^N$ in $\mathsection$\ref{secparbcomposite}.  We note that $\psi \in \Psi_2(U(4), \xi_+)$ implies that $\phi^N$ must be even.  The main result of the section is the following.

\begin{prop}

We have the bound $h^3_{(2)}(Y(n))^\star \ll n^9$.

\end{prop}

We define compact open subgroups $K' = \prod_p K'_p \subset U(2)(\A_f)$, $\widetilde{K}' = \prod_w \widetilde{K}_w' \subset GL(2,\A_{E,f})$, and $\widetilde{K} = \prod_w \widetilde{K}_w \subset GL(4,\A_{E,f})$.  We assume that $K'_p = U(2)(\Z_p)$ for all $p \notin S_f$, and likewise for the other groups.  The local components of these groups for $w|p \in S_f$ will be specified in the proof of Proposition \ref{parblocal}.  We define congruence subgroups $K'(n)$, etc. of these groups for $n$ relatively prime to $S_f$ in the usual way, and recall that $n$ will only be divisible by primes that split in $E$.

We let $\widetilde{P}$ be the standard parabolic subgroup of $GL(4,E)$ with Levi $\widetilde{L} = GL(2,E) \times GL(2,E)$, and let $P$ be the corresponding standard parabolic subgroup of $U(4)$.  We let $P'$ be the standard Borel subgroup of $U(2)$.

\subsection{Controlling a single parameter}

We fix an even parameter $\phi^N \in \tPhi_\text{sim}(2)$, and let $\psi \in \Phi(U(4), \xi_+)$ be the unique parameter with $\psi^N = \nu(2) \boxtimes \phi^N$.  We let $\phi^N$ correspond to a conjugate self-dual cuspidal automorphic representation $\mu$ of $GL(2,\A_E)$.  We assume that $\mu$ is tempered at all places; as before, this is done only for simplicity.  We let $\psi' \in \Psi(U(2), \xi_-)$ be the unique parameter with $\psi'^N = \phi^N$.  We shall prove the following bound for the finite part of the contribution of $\Pi_\psi$ to $h^3_{(2)}(Y(n))^\star$.

\begin{prop}
\label{parblocal}

There is a choice of $K'_p$ for $p \in S_f$, depending only on $K$, such that 

\be
\label{parblocalsum}
\sum_{\pi_f \in \Pi_{\psi,f}} \dim \pi_f^{K(n)} \ll [ K' : (K' \cap P'(\A_f)) K'(n) ] [ K : (K \cap P(\A_f)) K(n)] \sum_{\pi'_f \in \Pi_{\psi',f}} \dim \pi_f'^{K'(n)}.
\ee

\end{prop}

We begin the proof of Proposition \ref{parblocal} with Lemmas \ref{parbsplit}--\ref{parbKirilov} below, which control the left hand side of (\ref{parblocalsum}) in terms of $\mu$.

\begin{lemma}
\label{Speh}

Let $p \notin S_f$ be split in $E$, and let $w|p$.  Let $\Pi_{\psi_p} = \{ \pi_p \}$.  We have

\be
\label{parbsplitineq}
\dim \pi_p^{K_p(n)} \le [K_p : (K_p \cap P_p) K_p(n)] ( \dim \mu_w^{\widetilde{K}'_w(n)} )^2.
\ee

\end{lemma}

\begin{proof}

Under the identification $U(4)_p \simeq GL(4,E_w)$, $\pi_p$ is the Langlands quotient of the representation $\rho_{\psi_w}$ of $GL_4(E_w)$ induced from the representation $\mu_w(x_1) | \det(x_1) |^{1/2} \otimes \mu_w(x_2) | \det(x_2) |^{-1/2}$ of $\widetilde{P}_w$.  We have

\bes
\dim \pi_p^{K_p(n)} \le \dim \rho_{\psi_w}^{\widetilde{K}_w(n)}.
\ees
The restriction of $\rho_{\psi_w}$ to $\widetilde{K}_w$ is isomorphic to the induction of $\mu_w(x_1) \times \mu_w(x_2)$ from $\widetilde{K}_w \cap \widetilde{P}_w$ to $\widetilde{K}_w$.  We see that

\begin{align*}
\dim \rho_{\psi_w}^{\widetilde{K}_w(n)} & = [\widetilde{K}_w : (\widetilde{K}_w \cap \widetilde{P}_w) \widetilde{K}_w(n)] \dim (\mu_w \times \mu_w)^{\widetilde{L}_w \cap \widetilde{K}_w(n)} \\
& = [\widetilde{K}_w : (\widetilde{K}_w \cap \widetilde{P}_w) \widetilde{K}_w(n)] ( \dim \mu_w^{\widetilde{K}'_w(n)} )^2,
\end{align*}
which is equivalent to the lemma.

\end{proof}

We remove the square on the right hand side of (\ref{parbsplitineq}) using the following lemma.

\begin{lemma}
\label{parbKirilov}

If $p \notin S_f$ is split and $w|p$, we have

\bes
\dim \mu_w^{\widetilde{K}'_w(n)} \le [ \widetilde{K}'_w : (\widetilde{K}'_w \cap \widetilde{P}_w') \widetilde{K}'_w(n) ] = [ K'_p : (K'_p \cap P'_p) K'_p(n) ].
\ees

\end{lemma}

\begin{proof}

If $\mu_w$ is a principal series representation or a twist of Steinberg, this is immediate.  If $\mu_w$ is supercuspidal, this follows by examining the construction of supercuspidal representations given in $\mathsection$7.A. of \cite{Ge}.

\end{proof}

\begin{cor}
\label{parbsplit}

Let $p \notin S_f$ be split in $E$, and let $w|p$.  Let $\Pi_{\psi_p} = \{ \pi_p \}$.  We have

\bes
\dim \pi_p^{K_p(n)} \le [K_p : (K_p \cap P_p) K_p(n)] [ K'_p : (K'_p \cap P'_p) K'_p(n) ] \dim \mu_w^{\widetilde{K}'_w(n)}.
\ees

\end{cor}

\begin{lemma}
\label{parbnonsplit}

Let $p \notin S_f$ be nonsplit in $E$, and let $w|p$.  We have

\bes
\sum_{\pi_p \in \Pi_{\psi_p}} \dim \pi_p^{K_p} \le \dim \mu_w^{\widetilde{K}'_w}.
\ees

\end{lemma}

\begin{proof}

Identify $\widetilde{K}_w$ with a subgroup of $G(4)_p$.  The twisted fundamental lemma implies that the functions $1_{K_p}$ and $1_{\widetilde{K}_w \rtimes \theta}$ are related by transfer.  Applying Proposition \ref{charid} with $s = e$ gives

\bes
1_{K_p}^{U(4)}(\psi_p) = \sum_{\pi_p \in \Pi_{\psi_p}} \dim \pi_p^{K_p},
\ees
and combining this with Proposition \ref{charGL} and the twisted fundamental lemma gives

\bes
\sum_{\pi_p \in \Pi_{\psi_p}} \dim \pi_p^{K_p} = \tr( \widetilde{\pi}_{\psi_p}( 1_{\widetilde{K}_w \rtimes \theta}) ).
\ees
The twisted trace $\tr( \widetilde{\pi}_{\psi_p}( 1_{\widetilde{K}_w \rtimes \theta}) )$ is equal to the trace of $\widetilde{\pi}_{\psi_p}(\theta)$ on $\pi_{\psi_p}^{\widetilde{K}_w}$, so we have

\bes
\tr( \widetilde{\pi}_{\psi_p}( 1_{\widetilde{K}_w \rtimes \theta}) ) \le \dim \pi_{\psi_p}^{\widetilde{K}_w}.
\ees
Under the identification $G(4)_p \simeq GL(4,E_w)$, $\pi_{\psi_p}$ is the Langlands quotient of the representation $\rho_{\psi_w}$ induced from $\mu_w(x_1) | \det(x_1) |^{1/2} \otimes \mu_w(x_2) | \det(x_2) |^{-1/2}$.  We therefore have

\bes
\dim \pi_{\psi_p}^{\widetilde{K}_w} \le \dim \rho_{\psi_w}^{\widetilde{K}_w} \le \dim \mu_w^{\widetilde{K}'_w},
\ees
and the result follows.

\end{proof}

\begin{lemma}
\label{parbSf}

Let $p \in S_f$, and let $w|p$.  There is a choice of $\widetilde{K}'_w$, depending only on $K$, such that

\bes
\sum_{\pi_p \in \Pi_{\psi_p}} \dim \pi_p^{K_p} \ll \dim \mu_w^{\widetilde{K}'_w}.
\ees

\end{lemma}

\begin{proof}

Suppose that $p$ is nonsplit.  By \cite[Prop 3.1.1(b)]{Mo}, we may choose a function $\widetilde{1}_{K_p} \in \tcH_p(4)$ corresponding to $1_{K_p}$ under twisted transfer.  Reasoning as in Lemma \ref{parbnonsplit} gives

\bes
\sum_{\pi_p \in \Pi_{\psi_p}} \dim \pi_p^{K_p} = \text{vol}(K_p)^{-1} \tr( \widetilde{\pi}_{\psi_p}( \widetilde{1}_{K_p} )),
\ees
where $\text{vol}(K_p)$ denotes the volume of $K_p$ with respect to our choice of Haar measure on $U(4)_p$.  If we choose $\widetilde{K}_w \subset GL(4,E_w) \simeq G(4)_p$ to be a compact open subgroup such that $\widetilde{1}_{K_p}$ is bi-invariant under $\widetilde{K}_w$, we have

\bes
\tr( \widetilde{\pi}_{\psi_p}( \widetilde{1}_{K_p} )) \ll \dim \pi_{\psi_p}^{\widetilde{K}_w}.
\ees
Under the identification $G(4)_p \simeq GL(4,E_w)$, $\pi_{\psi_p}$ is the Langlands quotient of the representation $\rho_{\psi_w}$ induced from $\mu_w(x_1) | \det(x_1) |^{1/2} \otimes \mu_w(x_2) | \det(x_2) |^{-1/2}$.  Choose $\widetilde{K}'_w$ so that the product $\widetilde{K}'_w \times \widetilde{K}'_w$ is contained in $\widetilde{K}_w$.  We then have

\bes
\dim \pi_{\psi_p}^{\widetilde{K}_w} \le \dim \rho_{\psi_w}^{\widetilde{K}_w} \ll (\dim \mu_w^{\widetilde{K}'_w})^2.
\ees
Bounding $\dim \mu_w^{\widetilde{K}'_w}$ by a constant depending on $\widetilde{K}'_w$, and hence $K_p$, completes the proof for $p$ nonsplit.  The proof in the split case follows in exactly the same way using the explicit description of $\pi_p$.

\end{proof}

Let $S_{E/\Q}$ be a set of finite places of $E$ that contains exactly one place above every finite place of $\Q$.  Combining Corollary \ref{parbsplit}, Lemma \ref{parbnonsplit}, and Lemma \ref{parbSf} gives

\bes
\sum_{\pi \in \Pi_{\psi,f}} \dim \pi_f^{K(n)} \ll [ K' : ( K' \cap P'(\A_f)) K'(n) ] [ K : ( K \cap P(\A_f)) K(n) ] \prod_{w \in S_{E/\Q}} \dim \mu_w^{\widetilde{K}'_w(n)}.
\ees
Proposition \ref{parblocal} now follows from the lemma below.

\begin{lemma}

There is a choice of $K'_p$ for $p \in S_f$, depending only on $K$, such that

\bes
\prod_{w \in S_{E/\Q}} \dim \mu_w^{\widetilde{K}'_w(n)} \ll \sum_{\pi'_f \in \Pi_{\psi',f}} \dim \pi_f'^{K'(n)}.
\ees

\end{lemma}

\begin{proof}

We may factorise the right hand side as

\bes
\sum_{\pi'_f \in \Pi_{\psi',f}} \dim \pi_f'^{K'(n)} = \prod_p \sum_{\pi_p' \in \Pi_{\psi'_p}} \dim \pi_p'^{K_p'(n)}.
\ees
Let $p$ be an arbitrary prime, and $w | p$.  It suffices to show that

\be
\label{Kdescend}
\dim \mu_w^{\widetilde{K}'_w(n)} \le \sum_{\pi_p' \in \Pi_{\psi'_p}} \dim \pi_p'^{K_p'(n)}
\ee
if $p \notin S_f$, and that if $p \in S_f$ the same inequality holds with a constant factor depending only on $\widetilde{K}'$, and hence $K$.

If $p$ is split, then $\Pi_{\psi'_p}$ contains a single representation that is isomorphic to $\mu_w \otimes \chi_w^{-1}$ under the identification $U(2)_p \simeq GL(2,E_w)$, and (\ref{Kdescend}) is immediate.

Suppose that $p \notin S_f$ is nonsplit.  The definition of $\psi'_p$ implies that if $\xi_- : {}^LU(2)_p \rightarrow {}^LG(2)_p$, the representation of $G(2)_p \simeq GL(2,E_w)$ associated to $\xi_- \circ \psi_p' \in \Psi_p(2)$ is $\mu_w$.  We let $\widetilde{\mu}_w$ denote the canonical extension of $\mu_w$ to a representation of $\widetilde{G}^+(2)_p$, and identify $\widetilde{K}'_w$ with a subgroup of $G(2)_p$.  Proposition \ref{charGL} and the twisted fundamental lemma give

\be
\label{descendtrace}
\tr( \widetilde{\mu}_w(1_{\widetilde{K}'_w \rtimes \theta})) = \sum_{\pi_p' \in \Pi_{\psi'_p}} \tr( \pi_p'(1_{K_p'})) = \sum_{\pi_p' \in \Pi_{\psi'_p}} \dim \pi_p'^{K_p'}.
\ee
The left hand side of (\ref{descendtrace}) is equal to the trace of $\widetilde{\mu}_w(\theta)$ on $\mu_w^{\widetilde{K}_w'}$.  If $\dim \mu_w^{\widetilde{K}_w'} = 0$ then both sides of (\ref{descendtrace}) are 0, and (\ref{Kdescend}) holds.  If $\dim \mu_w^{\widetilde{K}_w'} = 1$, then $\theta^2 = 1$ implies that $\tr( \widetilde{\mu}_w(1_{\widetilde{K}'_w \rtimes \theta})) = \pm 1$.  Positivity implies that we must take the plus sign so that (\ref{Kdescend}) also holds.

Suppose that $p \in S_f$ is nonsplit, and suppose that the left hand side of (\ref{Kdescend}) is nonzero.  Up to twist, there are only finitely many possibilities for $\mu_w$ that are supercuspidal or Steinberg, and we may deal with these cases by simply choosing $K_p'$ so that (\ref{Kdescend}) is true in each case.  If $\mu_w$ is induced from a unitary character of the Borel, then $\Pi_{\psi'_p}$ is described explicitly in $\mathsection$11.4 of \cite{Ro} and (\ref{Kdescend}) follows easily from this description.

\end{proof}

\subsection{Summing over parameters}
\label{secparbsum}

We define $\Phi_\text{rel} \subset \widetilde{\Phi}_\text{sim}(2)$ to be the set of even parameters $\phi^N$ such that $\phi_\infty^N$ is given by

\bes
\phi_\infty^N : z \mapsto \left( \begin{array}{cc} z/\overline{z} & \\ & \overline{z}/z \end{array} \right).
\ees
It may be shown in the same way as Lemma \ref{archlocal} that if $\psi \in \Psi(U(4), \xi_+)$ satisfies $\psi^N = \nu(2) \boxtimes \phi^N$ with $\phi^N \in \widetilde{\Phi}_\text{sim}(2)$, and $\pi \in \Pi_{\psi_\infty}$ satisfies $H^*(\g,K; \pi) \neq 0$, then $\phi^N \in \Phi_\text{rel}$.  If $\phi^N \in \Phi_\text{rel}$ corresponds to the cusp form $\mu$, and $\chi_\infty$ is given by $\chi_\infty(z) = (z/\overline{z})^{1/2 + t}$ with $t \in \Z$, then $\mu_\infty \times \chi_\infty$ has infinitesimal character $(3/2 + t, -1/2+t; -3/2-t, 1/2-t) \in (\C^2 / S_2) \times (\C^2 / S_2)$.  Theorem 1.2 of \cite{Ca} then implies that $\mu$ is tempered at all places.  It follows from this discussion that

\be
\label{bsum1}
h^3_{(2)}(Y(n))^\star \ll \sum_{ \substack{ \psi^N = \nu(2) \boxtimes \phi^N \\ \phi^N \in \Phi_\text{rel} } } \sum_{\pi \in \Pi_\psi} \dim \pi_f^{K(n)}.
\ee

Applying Proposition \ref{parblocal} to the sum on the right hand side (and ignoring the factors $\#(\Pi_{\psi_\infty})$ as in Section \ref{secparasum}) gives

\be
\label{bsum2}
h^3_{(2)}(Y(n))^\star \ll [ K' : ( K' \cap P'(\A_f)) K'(n) ] [ K : ( K \cap P(\A_f)) K(n) ] \sum_{ \substack{ \psi' \in \Psi(U(2), \xi_-) \\ \psi'^N \in \Phi_\text{rel} } } \sum_{\pi' \in \Pi_{\psi'}} \dim \pi_f'^{K'(n)}.
\ee
The restriction on the infinitesimal characters of parameters in $\Phi_\text{rel}$ implies that there is a finite set of representations $\Xi_\infty$ of $U(2)_\infty$ such that if $\psi'^N \in \Phi_\text{rel}$, then all the representations in $\Pi_{\psi'_\infty}$ are in $\Xi_\infty$.  Because $\Phi_\text{rel}$ consists of simple generic parameters we have $\Pi_{\psi'} = \Pi_{\psi'}(\epsilon_{\psi'})$, and so every $\pi' \in \Pi_{\psi'}$ occurs in $L^2_\text{disc}(U(2)(\Q) \backslash U(2)(\A))$ with multiplicity one.  If we define $X(n) = U(2)(\Q) \backslash U(2)(\A) / K'(n)$, and let $m(\pi_\infty, X(n))$ denote the multiplicity as in Section \ref{secparasum}, this gives

\begin{align}
\notag
\sum_{ \substack{ \psi' \in \Psi(U(2), \xi_-) \\ \psi'^N \in \Phi_\text{rel} } } \sum_{\pi' \in \Pi_{\psi'}} \dim \pi_f'^{K'(n)} & \le \sum_{ \substack{ \pi' \in L^2_\text{disc}(U(2)(\Q) \backslash U(2)(\A)) \\ \pi'_\infty \in \Xi_\infty } } \dim \pi_f'^{K'(n)} \\
\notag
& = \sum_{\pi_\infty \in \Xi_\infty} m( \pi_\infty, X(n)) \\
\label{bsum3}
& \ll [K' : K'(n)].
\end{align}
Combining (\ref{bsum1})--(\ref{bsum3}) gives

\bes
h^3_{(2)}(Y(n))^\star \ll [ K' : ( K' \cap P'(\A_f)) K'(n) ] [ K : ( K \cap P(\A_f)) K(n) ] [K' : K'(n)],
\ees
and applying the formula for the order of $GL(N)$ over a finite field completes the proof.

\subsection{The case of composite $\phi^N$}
\label{secparbcomposite}

We now suppose that $\phi^N = \phi_1^N \boxplus \phi_2^N$, where $\phi_i^N \in \tPhi(1)$ correspond to conjugate self-dual characters $\mu_i$.  We may prove the following analogue of Proposition \ref{parblocal}.

\begin{prop}

There is a choice of $\widetilde{K}_{1,w}$ for $w|p \in S_f$, depending only on $K$, such that 

\bes
\sum_{\pi_f \in \Pi_{\psi,f}} \dim \pi_f^{K(n)} \ll [ K : (K \cap P(\A_f)) K(n)]  \dim \mu_1^{\widetilde{K}_1(n)} \dim \mu_2^{\widetilde{K}_1(n)}.
\ees

\end{prop}

Unlike Proposition \ref{parblocal}, this bound is sharp.  The reason for this is that the representation $\pi_{\psi_p}$ for split $p$ is equivalent to the induction of $(\mu_{1,w} \circ \det(x_1)) | \det(x_1)|^{1/2} \otimes (\mu_{2,w} \circ \det(x_2)) | \det(x_2)|^{-1/2}$ from $\widetilde{P}_w$ to $GL(4,E_w)$, and it is easy to give a sharp bound for the dimension of invariants under $\widetilde{K}_w(n)$, unlike the Speh representations considered in Lemma \ref{Speh}.  We obtain a bound of $n^{6 + \epsilon}$ for the contribution of these parameters to $h^3_{(2)}(Y(n))^\star$.

\end{document}